\documentclass[letterpaper, 10 pt, conference]{ieeeconf}  

\IEEEoverridecommandlockouts                              

\usepackage{amsmath} 
\usepackage{amssymb}  
\usepackage{color}
\usepackage{graphicx}
\DeclareMathOperator{\e}{\mathrm{e}}

\newtheorem{lemma}{Lemma}
\newtheorem{cor}{Corollary}

\title{\LARGE \bf
Impulsive feedback control for dosing applications
}

\author{Alexander Medvedev$^{1}$, Anton Proskurnikov$^{2}$, and Zhanybai T. Zhusubaliyev$^{3,4}$
\thanks{* This work was partially supported the Swedish Research Council under grant 2019-04451.}
\thanks{$^{1}$Department of Information Technology, 
        Uppsala University, SE-752 37 Uppsala, Sweden
        [{\tt\small alexander.medvedev@it.uu.se}]}%
\thanks{$^{2}$Department of Electronics and Telecommunications, Politecnico di Torino, Turin, Italy, 10129 [{\tt\small anton.p.1982@ieee.org}]}%
\thanks{$^{3}$Department of Computer Science, International Scientific Laboratory for
Dynamics of Non-Smooth Systems, Southwest State University, Kursk, Russia
        [{\tt\small zhanybai@hotmail.com}]}%
\thanks{$^{4}$Faculty of Mathematics and Information Technology, Osh State University, Lenin st. 331, 723500, Osh, Kyrgyzstan        }%
}

\begin{document}

\maketitle
\thispagestyle{empty}
\pagestyle{empty}

\begin{abstract}
This paper addresses a design procedure of pulse-modulated feedback control solving a dosing problem defined for implementation in a manual mode. Discrete dosing, as a control strategy, is characterized by exerting  control action on the plant in impulsive manner at certain time instants. Dosing 
 applications appear primarily in chemical industry and medicine where the control signal constitutes a sequence of (chemically or pharmacologically) active substance quantities (doses) administered to achieve a desired result. When the doses and the instants of their administration are adjusted as functions of some measured variable, a feedback control loop exhibiting nonlinear dynamics arises. The impulsive character of the interaction between the controller and the plant makes the resulting closed-loop system non-smooth.  Limitations of the control law with respect to control goals are discussed. An application of the approach at hand to neuromuscular blockade in closed-loop anesthesia is considered in a numerical example.

\end{abstract}

\section{INTRODUCTION}
Continuous processes controlled via impulsive actions appear often in nature and technology. A simple example of such a control strategy is a patient taking medication to alleviate the symptoms of a disease according to a prescribed drug dosing regimen.  This is an open-loop control law that can be turned into a closed-loop one when the regimen is adjusted with respect to the achieved therapeutic effect~\cite{J83}. In chemical batch processes, discrete dosing is utilized to control the composition and properties of the product by introducing additives~\cite{D19}. Another accessible example of impulsive feedback control is hot air ballooning where the balloon altitude is maintained by  switching on and off a burner that heats the air inside it.

Feedback drug dosing~\cite{C89} predominately  follows the common practice in process control, namely stabilization (or tracking) of the plant dynamics around a set point by means of a discrete controller with a constant sampling time. This concept is implemented, e.g., in the artificial pancreas~\cite{DHL14}, where the controlled variable is measured by a continuous glucose monitor and the actuator is a continuous insulin infusion pump. Yet, the biological pancreas controls the blood glucose level in impulsive manner~\cite{Evans09} which fact motivates the development of impulsive control laws for the artificial pancreas, e.g.~\cite{RGS20}.  

Impulsive dosing control enhances the patient safety, because the control signal can be described in the same terms as in manual drug administration, i.e., in terms of doses and inter-dose periods.  The main hazards of the artificial pancreas are, for instance,  over- and underdosing of insulin~\cite{BKK16}. In impulsive control, the minimal and maximal doses can be explicitly specified whereas flow rate-based drug administration requires calculation of those.

Another established application of feedback drug dosing is closed-loop anesthesia \cite{IMK15} that constitutes an instance of truly multi-variable control in a medical setting. Once again, the known solutions are derived  from process control and implemented by continuously manipulating the flow rates of the anesthetic drugs delivered by pumps.

Mathematical description of therapeutic effect produced by a drug dose is provided by pharmacokinetic-pharnacodynamic  (PK/PD) models developed within pharmacology (pharmacometrics). 
After being introduced into the organism, the drug is spread in the blood and tissues. The temporal evolution  of  the drug concentrations and the drug transport between different sites (compartments)  is described by a PK-model whereas the effect of the drug is described by a PD-model. The drug will eventually clear out from the organism and its effect subside thus necessitating the administration of new doses to sustain the pharmacotherapy.  Linear PK-models give normally good approximations of actually measured concentrations.  Since the state variables of  a PK-model correspond to compartmental concentrations, such models are positive. The PD-models have however to be nonlinear to reflect the finite number of receptors that   the drug molecules can bind to. This often results in Wiener-type block-oriented models with a Hill-function nonlinearity.

This paper deals with the design of a closed-loop  dosing algorithm implemented as a pulse-modulated feedback. The main contribution is twofold. First, feasibility analysis of a dosing problem to keep the control output of the PK/PD plant within a predefined interval is performed. Second, a controller design procedure that renders a closed-loop solution with a given impulse weight (dose) and period with respect to a positive Wiener model is proposed.

The paper is composed as follows. First a continuous nonlinear PK/PD   mathematical model of a widely used in surgery drug is defined. Then a pulse-modulated feedback controller is  introduced rendering the closed-loop dynamics hybrid. Further, the problem of controller design is formulated and related to the considered dosing application.  Feasibility of controller design under limitations on the dose and interdose interval are discussed; a necessary feasibility condition (Corollary~\ref{cor.1}), and also a simple sufficient condition (Corollary~\ref{cor.2}) are proven. Further, we establish a necessary and sufficient condition ensuring the existence of a 1-cycle of period $T$, which is compatible with output corridor constraint (Lemma~\ref{lem:suffic}); this condition requires to find the maximum and minimum of some transcendental function on an interval.
Finally, a controller design for the closed-loop system to exhibit a predefined periodic solution is performed and the resulting dynamics are simulated. To illustrate the behaviors of the system under plant uncertainty, bifurcation analysis is performed with respect to the individualizable PK/PD model parameters.

\section{NEUROMUSCULAR  BLOCKADE}\label{sec:NMB}
Neuromuscular blockade (NMB) is used in anesthesia to  relax skeletal muscles of the patient  and prevent involuntary movement that could impact surgery. The effect of a NMB drug is reliably quantified by clinically validated sensors (monitors). Train-of-four (ToF) ratio sensors based on  acceleromyography are currently considered  as ``clinical gold standard"  while novel sensor technologies are sought for \cite{L21}.  

In this paper, a Wiener model of the muscle relaxant {\it atracurium} presented in  \cite{SWM12} is utilized as a clinically relevant example of a PK/PD model.  The linear (PK) third-order part of the model is given by the transfer function
\begin{equation}\label{eq:lin_NMB}
W(s)=\frac{\bar Y(s)}{U(s)}=\frac{v_1 v_2 v_3 \alpha^3}{(s+v_1\alpha)(s+v_2\alpha)(s+v_3\alpha)},
\end{equation}
where $\bar Y (s)$ is the Laplace transform of the (unmeasurable)   linear dynamic part  output $\bar y(t)$ and $U(s)$ is the Laplace transform of the input.  The parameter $0<\alpha\le  0.1$ is  patient-specific  and estimated from data whereas the  rest of the transfer function parameters are fixed, $v_1=1$, $v_2=4$, and $v_3=10$.  The structure of the linear block  reflects the fact that the PK-part comprises three cascaded compartments, each of them possessing  first-order dynamics.

The output signal measured by the ToF-monitor  $y(t)$  characterizes the muscular function (100\%-NMB depth) and is related to the output of the linear block via a Hill function capturing the PD-part of the model
\begin{equation}\label{eq:nonlin_NMB}
y(t)= \varphi(\bar y)\triangleq\frac{100 C_{50}^\gamma}{ C_{50}^\gamma + {\bar y}^\gamma(t)}.
\end{equation}
The patient-specific PD-parameter  $0<\gamma\le 10$ is also estimated from data and model \eqref{eq:lin_NMB}, \eqref{eq:nonlin_NMB} is particularized to a patient by the pair $(\alpha, \gamma)$. 
The drug concentration producing 50\% of the maximum effect is  $C_{50}=3.2425$ $ \mu \mathrm{ g} \ \mathrm{ml}^{-1} $. 

Without medication, i.e for $\bar y(t)\equiv 0$, the model output is $y(t)=100\%$.   From the clinical data in \cite[Fig.~4]{SWM12}, the muscular function  is to be kept within the range $2\% \le y(t) \le 10\%$ throughout the surgery. Other NMB  sensors might produce a different effect measure. To promptly induce NMB, a bolus dose  of {\it atracurium} (calculated as 400--500 $\mu \mathrm{ g}$ for a $\mathrm{kg}$ weight)  is administered to the patient  in the beginning of surgery. This is expected to provide adequate relaxation for about $15-35~\mathrm{min}$. When the desired level of paralysis is achieved, the output $y(t)$ is maintained within the range by administering a maintenance dose each $15-25~\mathrm{min}$. 

Neuromuscular blocking agents are considered high-risk medicines because they also paralyze the muscle necessary for breathing. Notably, the patient safety framework in NMB is established in terms of individual doses and dose intervals. Therefore, an impulsive control law for an NMB agent administration complies well with the patient safety regulation.

\section{PULSE-MODULATED FEEDBACK CONTROL}
To facilitate the forthcoming  controller design, rewrite the linear part of the PK/PD-model expressed by \eqref{eq:lin_NMB} as a state-space Wiener model
\begin{equation}                            \label{eq:1}
\dot{x}(t) =Ax(t), \quad \bar y(t)=Cx(t),
\end{equation}
where $$A=\begin{bmatrix} -a_1 &0 &0 \\ g_1 & -a_2 &0 \\ 0 &g_2 &-a_3 \end{bmatrix}, B=\begin{bmatrix} 1 \\ 0 \\ 0\end{bmatrix}, C =\begin{bmatrix} 0 &0 &1 \end{bmatrix},
$$
$a_i=v_i\alpha>0$ are  positive distinct constants, and $g_1,g_2>0$ are chosen to yield $g_1g_2=v_1v_2v_3\alpha^3$.
The output $y$ of \eqref{eq:nonlin_NMB} is controlled by a pulse-modulated controller that inflicts instantaneous jumps on the state vector
\begin{equation}                             \label{eq:2}
\begin{aligned}
x(t_n^+) &= x(t_n^-) +\lambda_n B, \quad                                          
t_{n+1} =t_n+T_n,
\\ 
T_n &=\Phi(\bar y(t_n)), \quad \lambda_n=F(\bar y(t_n)).
\end{aligned}
\end{equation}
The minus and plus in a superscript denote the left-sided and
a right-sided limit, respectively.
 The amplitude modulation function $F(\cdot)$ and frequency modulation
function $\Phi(\cdot)$ are continuous, strictly monotonic for positive arguments; $F(\cdot)$ is non-increasing
and $\Phi(\cdot)$ is non-decreasing,
\begin{equation}                             \label{eq:2a}
0<\Phi_1\le \Phi(\cdot)\le\Phi_2, \quad 0<F_1\le F(\cdot)\le F_2,
\end{equation}
where $\Phi_1$, $\Phi_2$, $F_1$, $F_2$ are constant
numbers. Then \eqref{eq:2} constitutes a combined (frequency and amplitude) 
pulse modulation operator \cite{GC98} implementing an output feedback over \eqref{eq:1}. The time instant $t_n$ is called firing time 
and $\lambda_n$ represents the corresponding impulse weight. With respect to NMB model \eqref{eq:lin_NMB}, \eqref{eq:nonlin_NMB}, $\lambda_n$ is the drug dose administered  at the instant $t_n$, and $T_n $ is the interdose interval.  From \eqref{eq:2a}, it follows that a drug dose is administered at  earliest $\Phi_1$ and lastest $\Phi_2$   time units after the previous one. Importantly, the dose size is explicitly restricted to be not less than $F_1$ and not over $F_2$, which, for any controller operation mode, prevents both under- and overdosing.

The matrix $A$ is Metzler and Hurwitz stable, whereas $B,C$ satisfy the relationship
 \begin{equation}\label{CBLB}
 CB=0.
 \end{equation}
These properties are essential in PK/PD modeling and render  positivity of the elements of  $x$ that stand for drug concentrations in the model compartments, imply that the effect of the drug fading away in time, and, finally, guarantee that the effect of a drug dose does not immediately affect the output.

\section{IMPULSIVE CONTROL PROBLEM}\label{sec:problems}
Closed-loop system \eqref{eq:1}, \eqref{eq:2} coincides  with the structure of the impulsive Goodwin's oscillator, see \cite{Aut09}.
It is known from \cite{Aut09} that the impulsive Goodwin's oscillator does not have equilibria and exhibits only positive oscillatory periodic or non-periodic (e.g. chaotic or quasi-periodic) solutions. The dynamically simplest periodic solution type is so-called 1-cycle arising when the feedback controller generates a sequence of impulses of same amplitude that are distributed equidistantly over time, i.e  $\forall n: \lambda_n=\lambda, T_n=T$.

Two dosing problems have been solved within the impulsive feedback  control framework. One of them is motivated by the  necessity of  sustained administration  a certain maintenance dose after the state of NMB is induced. 
\paragraph*{Control problem 1~\cite{MPZh23,MPZh23a}} Given plant model \eqref{eq:1}, together with the dose $\lambda$ and the interdose interval $T$, design the modulation functions $\Phi(y)$ and $F(y)$  in \eqref{eq:2} so that closed-loop system \eqref{eq:1}, \eqref{eq:2}  exhibits an orbitally stable  1-cycle characterized by the desired parameters.

\paragraph*{Control problem 2~\cite{MPZh24}} Given plant model \eqref{eq:1} and the constants $0<y_{\min}<y_{\max}$, design the modulation functions $\Phi(y)$ and $F(y)$  in \eqref{eq:2} so that closed-loop system \eqref{eq:1}, \eqref{eq:2} exhibits an orbitally stable 1-cycle and the (steady) system output stays within the prescribed range
\begin{equation}\label{eq:range-open}
y_{\min}\le y(t)\le y_{\max}. 
\end{equation}

Control Problem~1 has been solved with respect to a linear time-invariant plant. The initial analysis of
Control Problem~2 provided in~\cite{MPZh24} addresses the more general Wiener model.

In contrast with classical process control, the design problems formulated above deal with a desired periodic solution rather than a desired stationary point.
With $X_n\triangleq x(t_n^-)$, the  discrete-time dynamics of $X_n$ are given by
\begin{equation} \label{eq:1c}
X_{n+1}= \e^{(t_{n+1}-t_n)A}(X_n+\lambda_n B),\quad n=0,1,\ldots
\end{equation}
The knowledge of $X_n$ allows to uniquely recover the trajectory on the interval $(t_n,t_{n+1})$ via~\eqref{eq:1} and~\eqref{eq:2}:
\begin{equation} \label{eq:1d}
x(t)=\e^{(t-t_n)A}(X_n+\lambda_n B),\quad t\in(t_n,t_{n+1}).
\end{equation}
Now, for a 1-cycle, by applying \eqref{eq:1c}, one has
\begin{equation}\label{eq:1-cycle}
    X=Q(X), \quad Q(\xi)\triangleq\mathrm{e}^{A\Phi(C\xi)}\left( \xi+ F(C\xi)B \right),
\end{equation}
that is, $X$ is the fixed point of the map $Q(\cdot)$. 
This completely characterizes~\cite{MPZh23,MPZh23a} the 1-cycle with the parameters $\Phi(\bar y_0)=T$, $F(\bar y_0)=\lambda$, where $\bar y_0=CX$. For such a cycle, $X=X_n\,\forall n$ is uniquely found from~\eqref{eq:1-cycle}, rewritten as
\[
X=\mathrm{e}^{A\Phi(\bar y_0)}\left(X+ F(\bar y_0)B \right)=\mathrm{e}^{TA}\left(X+ \lambda B \right)
\]
and entailing the relation
\begin{equation}\label{eq.1-cycle-X}
X=\lambda(I-\e^{TA})^{-1}B.     
\end{equation}
The remaining points of the 1-cycle orbit are recovered from~\eqref{eq:1d}.

The Jacobian of the map $Q(\cdot)$ defined in~\eqref{eq:1-cycle} at the fixed point $X$ is evaluated as
\begin{equation}\label{eq:closed_loop}
    Q^\prime(X)= \e^{AT}+KC,
\end{equation}
where
\begin{equation}\label{eq:gain}
    K= \begin{bmatrix}
    J &D
\end{bmatrix}\begin{bmatrix}
 F^\prime(\bar y_0)  \\  \Phi^\prime(\bar y_0) 
\end{bmatrix}, J=\e^{AT}B, D=AX
\end{equation}
The 1-cycle corresponding to $X$ is then orbitally stable if $Q^\prime(X)$ is Schur stable. 
\section{FEASIBILITY ANALYSIS}

Feasibility analysis of Control problem~2 defined in Section~\ref{sec:problems} is performed here for a case when the cycle parameters are constrained to certain intervals $T_{\min}\le T \le T_{\max}$ and $0\le \lambda \le \lambda_{\max}$ due to safety or implementation reasons. 

\subsection{Necessary condition}

One can expect that, to reach a desired therapeutic effect, the parameters of 1-cycle $(\lambda,T)$ should satisfy some constraints to be compatible with the output corridor in~\eqref{eq:range-open}. If the drug is administered rarely, then each dose needs to be large enough to sustain the effect. On the other hand, more frequent drug administration requires lesser dosing. In fact, a more general statement can be proved for a general and, possibly, aperiodic solution of~\eqref{eq:1},\eqref{eq:2}.
\begin{lemma}\label{lem.necess}
Consider a solution to  impulsive system~\eqref{eq:1},\eqref{eq:2} such that $\lambda_n\leq \lambda^*$ and $T_n\geq T_*>0$ for all $n$. Then, the ultimate lower output bound  holds
\begin{equation}\label{eq.upper-bound1}
\limsup_{t\to\infty}\bar y(t)\leq 
\frac{g_1g_2}{a_2a_3}\frac{\lambda^*}{1-\e^{-a_1T_*}}.
\end{equation}
Similarly, if $\lambda_n\geq \lambda_*\geq 0$ and $T_n\leq T^*$ for all $n$, then
\begin{equation}\label{eq.lower-bound1}
\liminf_{t\to\infty}\bar y(t)\geq 
\frac{g_1g_2}{a_2a_3}\frac{\lambda_*\e^{-a_1T^*}}{1-\e^{-a_1T^*}}.
\end{equation}
\end{lemma}
\begin{proof}
The proof retraces the arguments in the proof of~\cite[Proposition~1]{ZCM12b}. A closer look at the proof in~\cite{ZCM12b} reveals that~\eqref{eq.upper-bound1} follows from the inequalities $\lambda_n=F(CX_n)\leq\lambda^*\,\forall n$
and $t_{n+1}-t_n=T_n=\Phi(CX_n)\geq T_*\,\forall n$ (which in~\cite{ZCM12b} are replaced by more conservative estimates
$\lambda_n\leq F_2$ and $T_n\geq\Phi_1$). Similarly,~\eqref{eq.upper-bound1} is entailed by 
the inequalities $\lambda_n\geq\lambda_*$
and $T_n\leq T^*$ (which in~\cite{ZCM12b} are replaced by estimates
$\lambda_n\geq F_1$ and $T_n\leq\Phi_2$).
\end{proof}

Lemma~\ref{lem.necess} does not involve the fixed point $X$ of the map $Q$ and, therefore, applies also in transient mode. 
Applying the result to 1-cycle, the following corollary is obtained.
\begin{cor}\label{cor.1}
Assume that there exists a 1-cycle with the parameters $\lambda_n\equiv \lambda$ and $T_n\equiv T$ that satisfies~\eqref{eq:range-open}. 
Then,
\begin{equation}\label{eq:necess_bounds}
    \frac{a_2a_3}{g_1g_2}\bar y_{\min}(1-\e^{-a_1T})\leq\lambda\leq \frac{a_2a_3}{g_1g_2}\bar y_{\max}(\e^{a_1T}-1),
\end{equation}
where, by definition,\footnote{Recall that function $\varphi$ is decreasing, hence, $\bar y_{\min}<\bar y_{\max}$.}
$\bar y_{\min}\triangleq\varphi^{-1}(y_{\max})$ and $\bar y_{\max}\triangleq\varphi^{-1}(y_{\min})$. In particular,
if $\lambda\leq\lambda_{\max}$ and $T\geq T_{\min}$, then
\[
\lambda_{\max}\geq \frac{a_2a_3}{g_1g_2}\bar y_{\min}(1-\e^{-a_1T_{\min}}).
\]
\end{cor}
\begin{proof}
The first statement follows immediately from Lemma~\ref{lem.necess}, assuming that $\lambda_*=\lambda^*=\lambda$ and $T^*=T_*=T$.
Since the output satisfies~\eqref{eq:range-open}, one obtains
\[
\begin{gathered}
\bar y_{\min}\leq \limsup_{t\to\infty}\bar y(t)\leq\frac{g_1g_2}{a_2a_3}\frac{\lambda}{1-\e^{-a_1T}},\\
\bar y_{\max}\geq \liminf_{t\to\infty}\bar y(t)\geq\frac{g_1g_2}{a_2a_3}\frac{\lambda}{\e^{a_1T}-1},
\end{gathered}
\]
which proves the first statement. The second statement is now straightforward from~\eqref{eq:necess_bounds}.
\end{proof}

\subsection{Sufficient condition}

We will now obtain two conditions, ensuring that, for a given period $T$, an amplitude $\lambda\in [0,\lambda_{\max}]$ exists such that the 1-cycle with parameters $(\lambda,T)$ satisfies output constraint~\eqref{eq:range-open}. Notice that these conditions \emph{do not} imply  stability of the 1-cycle, which needs to be tested separately.

The first condition resembles Corollary~\ref{cor.1} and is also based on Lemma~\ref{lem.necess}.
\begin{cor}\label{cor.2}
If $T$ and $\lambda_{\max}$ satisfy the inequalities
\begin{equation}\label{eq:1-cycle-suffic-simple1}
\begin{gathered}
\e^{a_1T}\bar y_{\min}<\bar y_{\max},\\
\lambda_{\max}\geq \frac{a_2a_3}{g_1g_2}\bar y_{\min}(\e^{a_1T}-1),
\end{gathered}
\end{equation}
then there exists $\lambda\in[0,\lambda_{\max}]$ such that the 1-cycle with parameters $(\lambda,T)$ obeys~\eqref{eq:range-open}.
\end{cor}
\begin{proof}
Recall that the output of any 1-cycle $\bar y(t)$ is a continuous periodic function, hence,
\[
\limsup_{t\to\infty}\bar y(t)=\max_{t\geq 0}\bar y(t),\quad \liminf_{t\to\infty}\bar y(t)=\min_{t\geq 0}\bar y(t).
\]
Choosing now $\lambda$ as follows
\[
\lambda=\frac{a_2a_3}{g_1g_2}\bar y_{\min}(\e^{a_1T}-1),
\]
one guarantees that $0\leq\lambda\leq\lambda_{\max}$ and, by virtue of~\eqref{eq.lower-bound1} (with $\lambda_*=\lambda^*=\lambda$ and $T^*=T_*=T$), one has $\min_{t\geq 0}\bar y(t)\geq \bar y_{\min}$. Applying~\eqref{eq.lower-bound1}, one shows that
\[
\max_{t\geq 0}\bar y(t)\leq \bar y_{\min}e^{a_1T}\leq y_{\max}.
\]
Hence, we have constructed 1-cycle that obeys~\eqref{eq:range-open}.
\end{proof}

An apparent disadvantage of the condition in Corollary~\ref{cor.2} is the restriction on $T$  that is often violated in practice.
The following lemma gives a more subtle condition, being, in fact, both necessary and sufficient condition, yet requires computing the extreme values of a non-convex transcendental function on a line segment.
\begin{lemma}\label{lem:suffic}
The 1-cycle with parameters $(\lambda,T)$ satisfies  output constraint~\eqref{eq:range-open} if and only if 
\begin{equation}\label{eq:1-cycle-suffic}
\begin{gathered}
\bar y_{\min}\leq\lambda\xi_T(\tau)\leq\bar y_{\max}\quad\forall\tau\in[0,T],\\
\xi_T(\tau)\triangleq Ce^{\tau A}(I-e^{TA})^{-1}B=Ce^{\tau A}\xi_T(0).
\end{gathered}
\end{equation}
In particular, a cycle with period $T$ satisfying~\eqref{eq:range-open} exists if and only if
\begin{equation}\label{eq:1-cycle-suffic2}
\frac{\max_{\tau\in[0,T]}\xi_T(\tau)}{\min_{\tau\in[0,T]}\xi_T(\tau)}\leq\frac{\bar y_{\max}}{\bar y_{\min}},
\end{equation}
and the minimal dose $\lambda$ that is needed to satisfy~\eqref{eq:range-open} is
\begin{equation}
\lambda_{\mathrm{opt}}=\frac{\bar y_{\min}}{\min_{\tau\in[0,T]}\xi_T(\tau)}.
\end{equation}
\end{lemma}
\begin{proof}
Recall that 1-cycle is a solution satisfying $t_{n+1}-t_n=T_n\equiv T$ and $X_n=X(t_n^-)\equiv X$, where $X=X(\lambda,T)$ is found from~\eqref{eq.1-cycle-X}. Substituting this in~\eqref{eq:1d}, one finds that
\[
y(t_n+\tau)=Cx(t_n+\tau)=\lambda Ce^{\tau A}(I-e^{TA})^{-1}B=\lambda\xi_T(\tau)
\]
for all $\tau\in(0,T)$. Furthermore, for 1-cycle (as well as any other solution), the coordinate
$y=Cx(t)=x_3(t)$ remains continuous, because $CB=0$. Hence, the latter relation holds also for $\tau=0$ and $\tau=T$.
Since the output $y(t)$ of the 1-cycle is a $T$-periodic function,  condition~\eqref{eq:range-open} is equivalent to~\eqref{eq:1-cycle-suffic}. This proves the first part of the lemma.

The second statement of the lemma follows from the first one. If the first condition in~\eqref{eq:1-cycle-suffic} is satisfied, then
\[
\frac{\lambda\max_{\tau\in[0,T]}\xi_T(\tau)}{\lambda\min_{\tau\in[0,T]}\xi_T(\tau)}\leq \frac{\bar y_{\max}}{\bar y_{\min}},
\]
which is equivalent to~\eqref{eq:1-cycle-suffic2}; this proves the ``only if'' part. Furthermore, it is obvious that~\eqref{eq:range-open} and~\eqref{eq:1-cycle-suffic2} can hold only when $\lambda\geq\lambda_{\mathrm{opt}}$.
To prove the ``if'' part, choose
$\lambda=\lambda_{\mathrm{opt}}$. Then, by construction, the output of the 1-cycle $\bar y(t)=\lambda\xi_T(t)$ attains the minimal value
\[
\min_{t\in[0,T]}\bar y(t)=\lambda_{\mathrm{opt}}\min_{t\in[0,T]]}\xi_T(t)=\bar y_{\min},
\]
whereas its maximum, in view of~\eqref{eq:1-cycle-suffic2}, does not exceed $\bar y_{\max}$. Recalling that $\bar y(t)$ is a continuous $T$-periodic function, it satisfies~\eqref{eq:range-open}, which finishes the proof.
\end{proof}

\section{DESIGN}

In this section, the output nonlinearity is assumed to be incorporated in the modulation functions, i.e.
\begin{equation}\label{eq:modulation}
F(\bar y)=  (\bar F\circ \varphi)(\bar y), \Phi(\bar y)= (\bar \Phi\circ \varphi)(\bar y),
\end{equation}
where $\circ$ is the composition of functions.

To illustrate the use of impulsive feedback control law \eqref{eq:2} in NMB, consider model \eqref{eq:lin_NMB},  \eqref{eq:nonlin_NMB} where the individualization parameters are set to the mean population values $\alpha=0.0374$, $\gamma=2.6677$, and $C_{50}=3.2425$, \cite{RMM14}.

Control Problem~2 formulated in Section~\ref{sec:problems} is treated with respect to this model in \cite{MPZh24}. Being designed to keep the Wiener model output in the interval $2\% \le y(t) \le 10\%$, the controller renders a 1-cycle with the drug dose $\lambda=415.8412$ and the interdose interval $T=37.3834$. With respect to the values typical to clinical practice, the dose of {\it atracurium} is too high for NMB maintenance and administered too seldom. High doses of NMB agents are associated with adverse side effects and administering the drug seldom weakens the effect of the feedback. 

To pursue a more realistic dosing regimen, select $\lambda=300, T=20$ and solve Control Problem~1.
First, check the necessary condition in \eqref{eq:necess_bounds} that yields
\[
190.1695 <\lambda < 476.9292,
\]
and is satisfied. Notice also that this condition is also satisfied for the solution obtained in \cite{MPZh24}, although results in a dose outside of the clinically established interval.

The fixed point corresponding to the desired 1-cycle is obtained from \eqref{eq:1-cycle} and given by
\begin{equation}\label{eq:X}
    X^\intercal=\begin{bmatrix} 269.5974 &84.5819 &13.6249 \end{bmatrix},
\end{equation}
and thus implying $\bar y_0=13.6249$.

With 
\begin{equation}\label{eq:F_Phi_fixed_point}
    F^\prime(\bar y_0)=-0.15,\quad \Phi^\prime(\bar y_0)=0.29,
\end{equation}
 the eigenvalues of the Jacobian in \eqref{eq:closed_loop} are $\sigma(Q^\prime(X))=\{0.2288, 0.1863, 0.0003\}$, which guarantees local stability of the fixed point. As pointed out in \cite{MPZh23}, the slopes of the modulation functions evaluated at the fixed point play the role of feedback gains in the problem of Schur-stabilization of a linear time-invariant system by  output feedback. Then all the design methods developed for the latter are viable here. For instance, in \cite{MPZh23a}, an approach making use of Bilinear Matrix Inequality is utilized. Yet, only local stability in vicinity of the fixed point can be secured in this way.

In \eqref{eq:modulation}, the functions $\bar F (\cdot), \bar \Phi (\cdot)$ represent the design degrees of freedom and have to guarantee the desired characteristics of the 1-cycle in the closed-loop system as well as its (orbital) stability. Select these modulation functions as piecewise affine, i.e.
\begin{align*}
    \bar \Phi (\xi)= \begin{cases} \Phi_2 &   \Phi_2 < k_2\xi +k_1, \\
     k_2\xi +k_1 & \Phi_1 \le  k_2\xi +k_1 \le \Phi_2, \\
    \Phi_1  &  k_2\xi +k_1 < \Phi_1, 
     \end{cases}
\end{align*}
\begin{align*}
    \bar F (\xi)= \begin{cases} F_1 &  k_4\xi +k_3< F_1, \\
     k_4\xi +k_3 & F_1 \le k_4\xi +k_3 \le F_2, \\
    F_2 & F_2 <k_4\xi +k_3.
     \end{cases}
\end{align*}
Then, impulsive feedback \eqref{eq:2} is completely defined by the constants $k_i, i=1,\dots, 4$.
Assuming that $F^\prime(\cdot)$ and $\Phi^\prime(\cdot)$ are in linear interval and applying the chain rule gives
\begin{align}\label{eq:F_prime_Phi_prime}
     F^\prime(\bar y_0)&=\bar F^\prime(\bar y_0)\varphi^\prime(\bar y_0)= k_4 \varphi^\prime(\bar y_0),\\ 
     \Phi^\prime(\bar y_0)&= \bar \Phi^\prime(\bar y_0)\varphi^\prime(\bar y_0)= k_2 \varphi^\prime(\bar y_0), \nonumber
\end{align}
where
\[
\varphi^\prime(\xi)= -\frac{\gamma 100 C_{50}^\gamma  \xi^{\gamma-1}}{ {(C_{50}^\gamma + {\xi}^\gamma)}^2}, \quad \varphi^\prime(\bar y_0)=-0.4073.
\]
Thus, from \eqref{eq:F_Phi_fixed_point}  and \eqref{eq:F_prime_Phi_prime}, it follows that
\[ k_2=-0.7119 , \quad k_4=0.3682.
\]
Now, the rest of the coefficients of the modulation functions are obtained from
\begin{align*}
    F(\bar y_{0})&= (\bar F \circ \varphi)(\bar y_{0})=\bar F ( \varphi(\bar y_{0}))=k_4 \varphi(\bar y_{0})+k_3= \lambda,\\
    \Phi(\bar y_{0})&= (\bar \Phi \circ \varphi)(\bar y_{0})= \bar \Phi (\varphi(\bar y_{0}))= k_2 \varphi(\bar y_{0})+k_1=T.
\end{align*}
that yield 
\[
k_1=21.5133 , \quad k_3=299.2173. 
\]
\begin{figure}[ht]
\centering 
\includegraphics[width=0.9\linewidth]{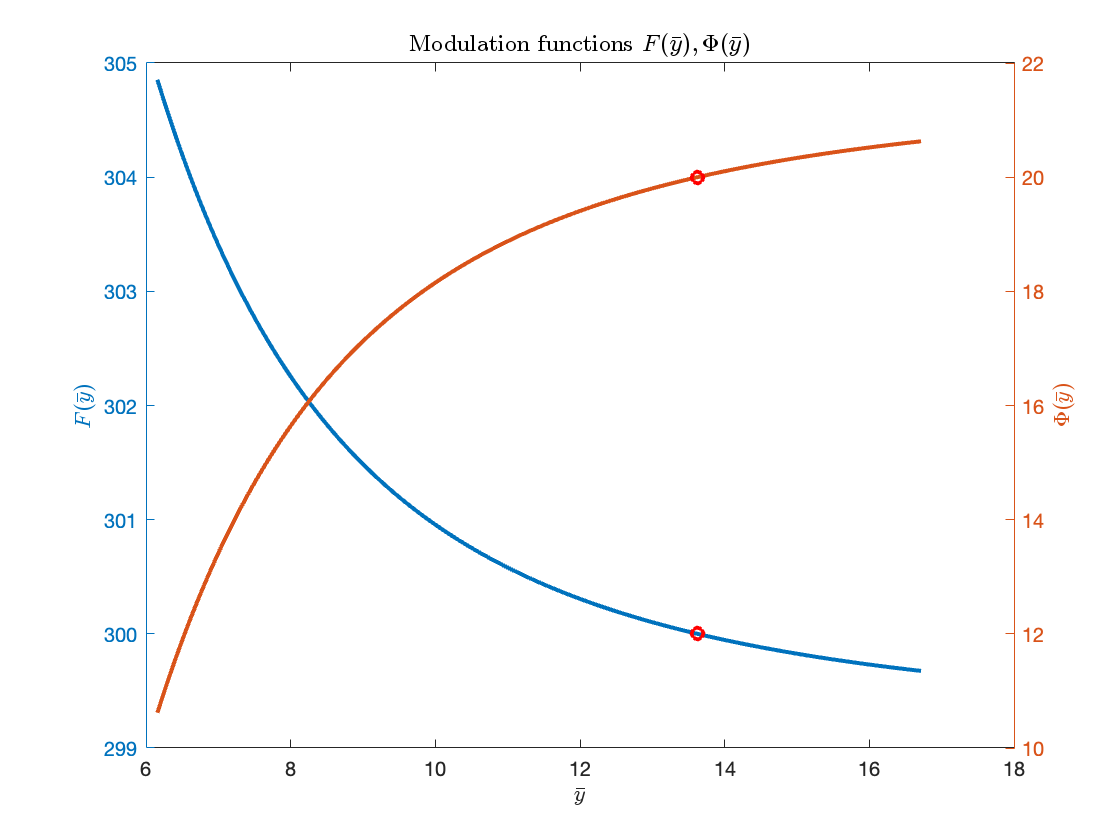}
\caption{The designed modulation functions $F(\bar y)$ (blue) and $\Phi(\bar y) $ (red). The desired cycle parameters $F(\bar y_0)=\lambda=300, \Phi(\bar y_0)=T=20$ are marked by red dot.}\label{fig:F_Phi}
\end{figure}
\begin{figure}[ht]
\centering 
\includegraphics[width=0.9\linewidth]{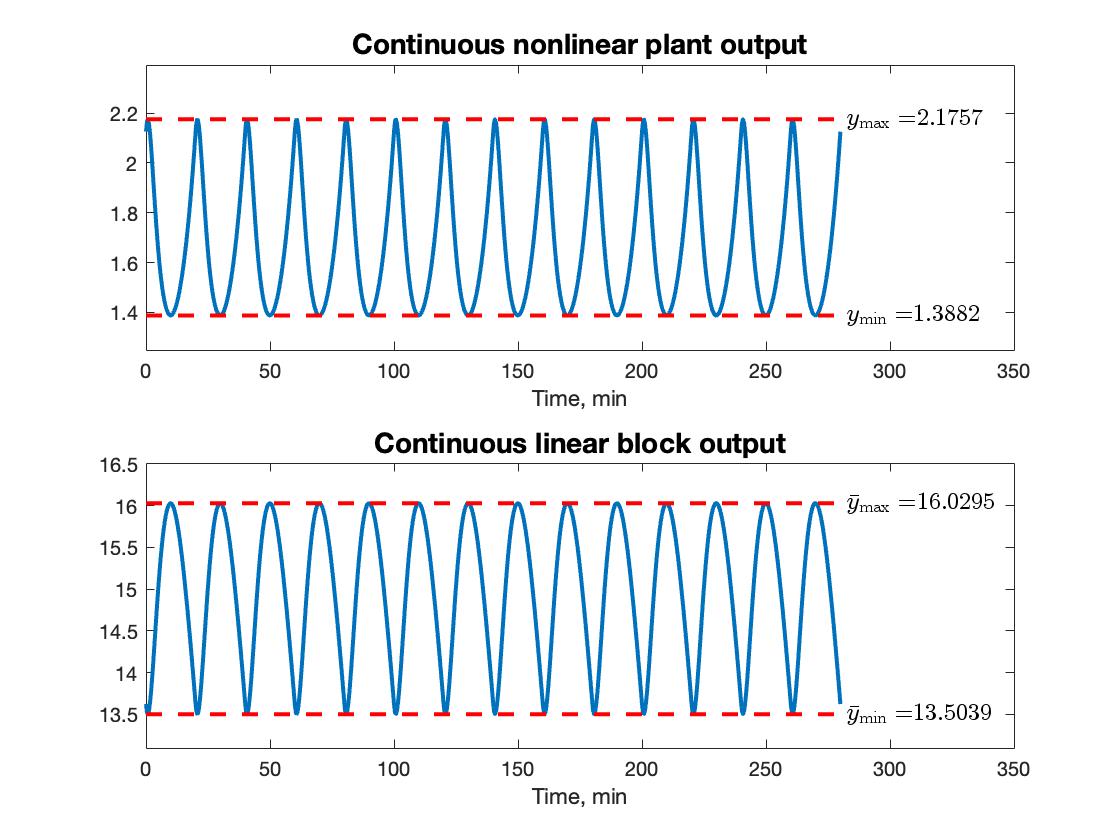}
\caption{The nonlinear output $y(t)$ (upper plot) and the linear block output $\bar y(t)$ (lower plot) in the designed 1-cycle. }\label{fig:1-cycle}
\end{figure}
\begin{figure}[ht]
\centering 
\includegraphics[width=0.9\linewidth]{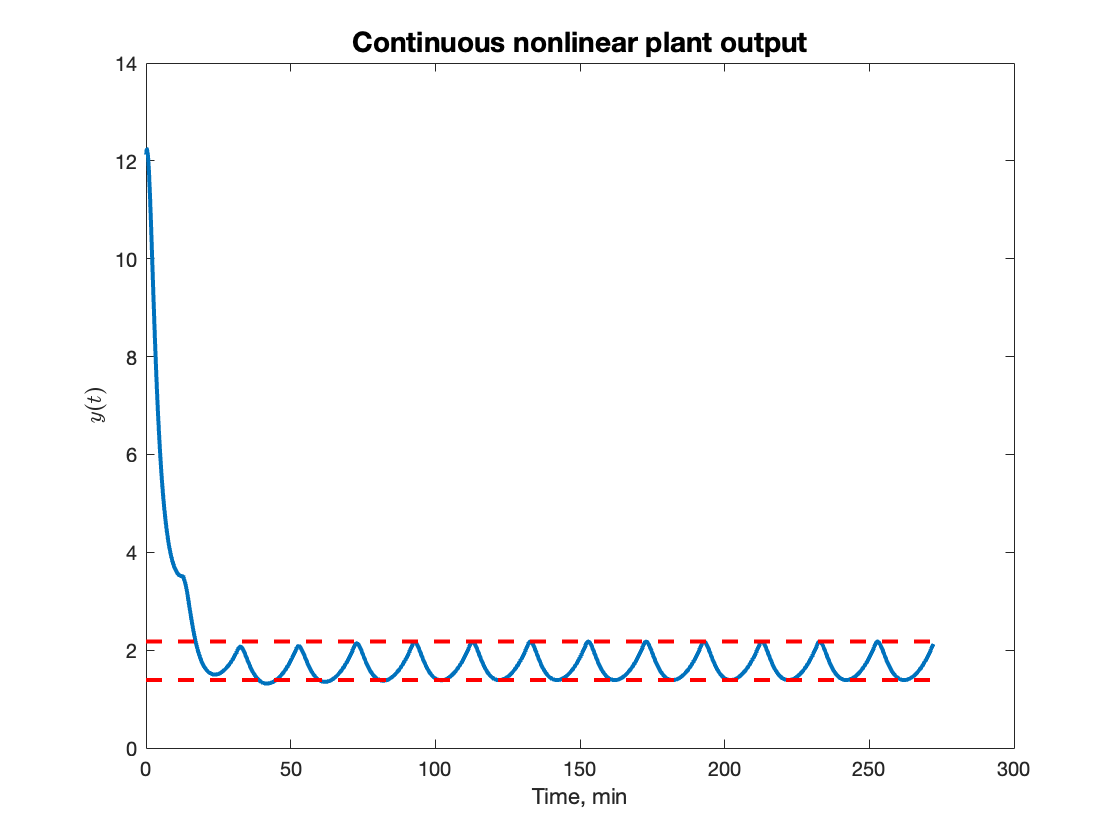}
\caption{The nonlinear output $y(t)$ in a transient to the designed 1-cycle. The 1-cycle is attractive. }\label{fig:transient}
\end{figure}
The resulting modulation functions are depicted in Fig.~\ref{fig:F_Phi}. Notice that \eqref{eq:F_prime_Phi_prime} is valid only when $F_i, \Phi_i, i=1,2$ are outside of the range of $\varphi(\cdot)$, which can always be satisfied.

Being initiated at the fixed point $X$, closed-loop system \eqref{eq:1}, \eqref{eq:2} exhibits a 1-cycle with the desired parameters $\lambda=300$, $T=20$, see Fig.~\ref{fig:1-cycle}. The maximal and minimal values of $y(t)$ and $\bar y(t)$ are calculated by applying the result of \cite[Proposition~2]{MPZh24}. Notably, selecting the mean values of the dose and interdose interval of {\it atracurium} resulted in the highest value of the muscular function $y(t)$ only slightly over the lower bound of the established clinical interval (i.e. $2\%$). Therefore, the actual value of $y(t)$ is  significantly below the lower bound of the clinical interval most of the time. Compared to the dosing regimen solution obtained in \cite{MPZh24} to keep $y(t)$ within a clinically suitable interval, the present strategy corresponds to drug overdosing. 

The designed 1-cycle is attractive. In Fig.~\ref{fig:transient}, a transient response of closed-loop system \eqref{eq:1}, \eqref{eq:2} is depicted. After the induction dose, an additional dose is needed to reach the stationary periodic orbit. 
\section{BIFURCATION ANALYSIS}
The dynamics of closed-loop system \eqref{eq:1}, \eqref{eq:2} are highly nonlinear \cite{ZCM12b} whereas the impulsive controller design in the previous section is based on a linearization in vicinity of a fixed point. To discern behaviors of the closed-loop dynamics under uncertainty in the plant parameters $\alpha, \gamma$, bifurcation analysis is performed.

The result of a bifurcation analysis for $0.0274<\alpha<0.04824$ is presented in Fig.~\ref{fig3}. The values of the rest of the NMB model parameters are kept as specified in Section~\ref{sec:NMB}.  At the point  $\alpha=\alpha_0$, 
$\alpha_0=0.0374$,  there is a stable fixed point specified by \eqref{eq:X} that undergoes a classical period-doubling bifurcation as  the parameter $\alpha$ is  increased. With further increase in the value of $\alpha$, the unstable  fixed point becomes a stable one in the border collision flip (or period-doubling) bifurcation.  As seen from the Fig.~\ref{fig3}, one of the periodic points of the 2-cycle collides with one of the borders. As a result, the 2-cycle changes its type. Such
a transition is known as ``persistence border collision''. 

With $\gamma$ in the range $1.403 <\gamma<5.5619$, there is only a stable fixed point that tends to the limit value $X_*^{\intercal}=\lbrack 242.126,\;76.5507,\;12.38252\rbrack$.
\begin{figure}[h!]
\centering
\parbox[c]{0.9\linewidth}{\includegraphics[width=\linewidth]{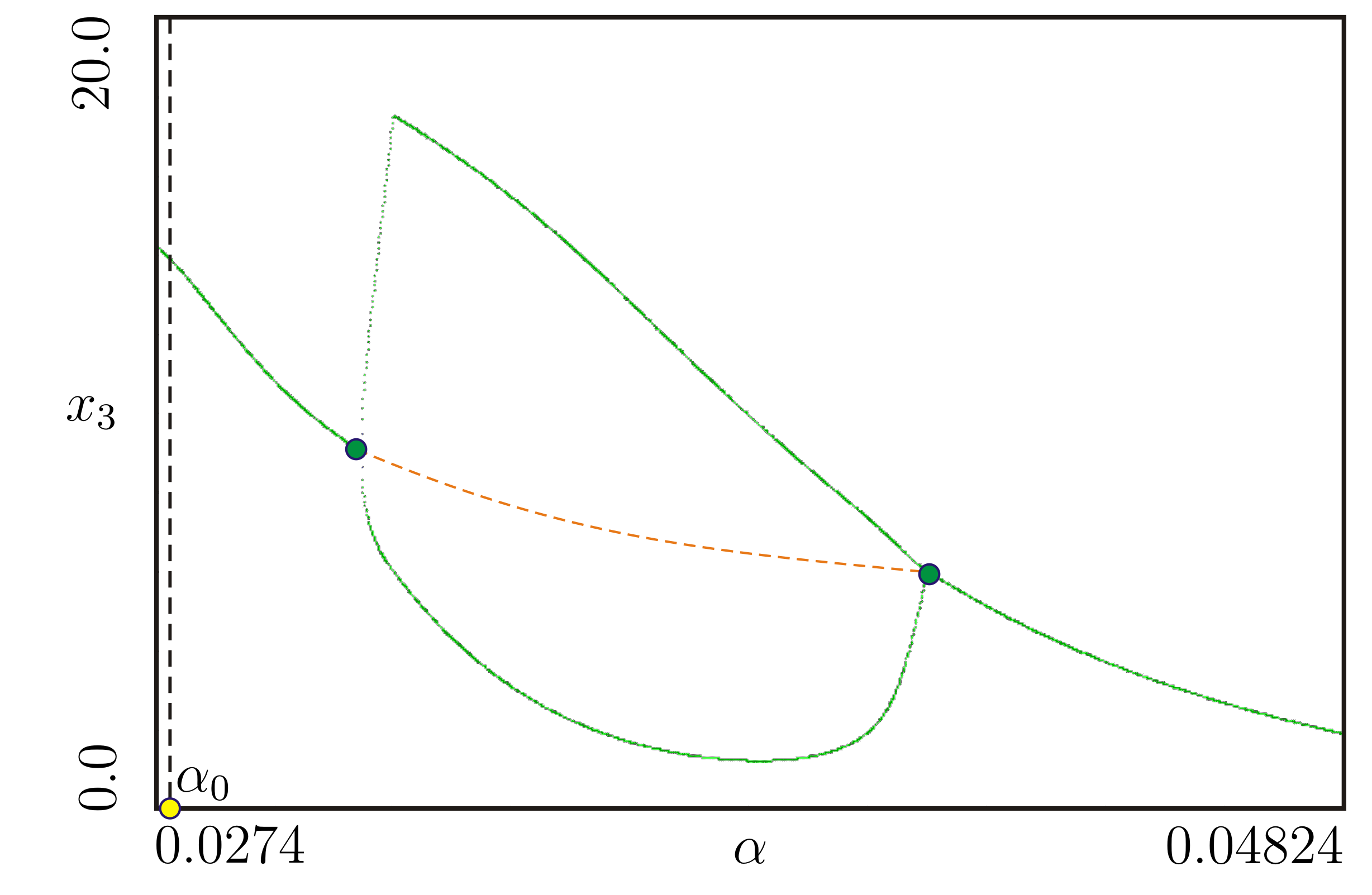}}\\
\caption{\label{fig3} Bifurcation diagram for
$0.0274<\alpha<0.04824$, $\alpha_0=0.0374$.}
\end{figure}
\section{CONCLUSIONS}
A pulse-modulated dosing feedback system suitable for applications in e.g. medicine and chemistry is considered. It can be designed directly from a description of (manual) dosing regimen, which facilitates controller validation and certification. This is in contrast with the conventional feedback control systems that are used in process control and designed to maximize closed-loop performance expressed as a function of the control error. The nonlinear and non-smooth dynamics arising in pulse-modulated control require complementing the design procedure with bifurcation analysis to ensure consistent system behavior under plant parameter uncertainty.

\bibliography{refs,observer}
\bibliographystyle{IEEEtran}
\end{document}